\documentclass[10pt,a4paper,twoside,reqno]{amsart}

\usepackage[a4paper,%bindingoffset=0.2in,
left=1in,right=1in,top=1.2in,bottom=1.2in,%
footskip=.25in]{geometry}

\usepackage{bm}
\usepackage{latexsym, amsmath, amstext, amssymb, amsfonts, amscd, bm, array, multirow, amsbsy, mathrsfs, stmaryrd}
\usepackage{amsthm}
\usepackage{t1enc}
\usepackage[mathscr]{eucal}
\usepackage{indentfirst}
\usepackage{pb-diagram}
\usepackage{graphicx}
\usepackage{fancyhdr}
\usepackage{fancybox}
\usepackage{color}
\usepackage{tikz-cd}
\usetikzlibrary{arrows}
\usepackage[all]{xy}
\usepackage{hyperref}
\usepackage{tikz}
\usepackage{xparse}
\hypersetup{colorlinks=false,pdfborderstyle={/S/U/W 0}}
\usetikzlibrary{matrix}
\usepackage{upgreek}
\usepackage[utf8]{inputenc}
\usepackage[T1]{fontenc}
\usepackage{bbm}
\usepackage{enumitem}
\setlist[enumerate]{leftmargin=*,noitemsep}
\setlist[itemize]{leftmargin=*,noitemsep}
\usepackage{float}
\usepackage{soul}

\newcommand{\doi}[1]{\url{https://doi.org/#1}}
\usepackage[noadjust]{cite}

\makeatletter
\def\@defaultbiblabelstyle#1{[#1]}
\makeatother

\usepackage{url}
\theoremstyle{plain}
\newtheorem{thm}{Theorem}[section]

\newtheorem{prop}[thm]{Proposition}

\newtheorem{lemma}[thm]{Lemma}

\theoremstyle{definition}
\newtheorem{definition}[thm]{Definition}

\theoremstyle{remark}
\newtheorem{remark}[thm]{Remark}

\newtheorem*{ack}{Acknowledgements}

\newcommand{\End}{\mathrm{End}}

\DeclareFontFamily{U}{rsf}{}
\DeclareFontShape{U}{rsf}{m}{n}{<5> <6> rsfs5 <7> <8> <9> rsfs7 <10-> rsfs10}{}
\DeclareMathAlphabet\Scr{U}{rsf}{m}{n}

\def\N{\mathbb{N}}

\def\C{\mathbb{C}}
\def\R{\mathbb{R}}

\def\Id{\mathrm{Id}}

\newcommand{\be}{\begin{equation*}}
\newcommand{\ee}{\end{equation*}}
\newcommand{\ben}{\begin{equation}}
\newcommand{\een}{\end{equation}}
\newcommand{\beqa}{\begin{eqnarray*}}
\newcommand{\eeqa}{\end{eqnarray*}}
\newcommand{\beqan}{\begin{eqnarray}}
\newcommand{\eeqan}{\end{eqnarray}}

\def\scB{\Scr B}

\def\Spin{\mathrm{Spin}}

\def\Spin{\mathrm{Spin}}

\def\U{\mathrm{U}}

\def\G_2{\mathrm{G_2}}

\def\G{\mathrm{G}}

\def\R{\mathbb{R}}

\def\mCl{\mathbb{C}\mathrm{l}}

\def\cL{\mathcal{L}}

\newcommand{\escal}[1]{\langle#1\rangle}

\renewcommand{\ker}{\mathrm{Ker}}

\newcolumntype{P}[1]{>{\centering\arraybackslash}p{#1}}

\usepackage[math]{anttor}
\usepackage{orcidlink}

\allowdisplaybreaks
\usepackage{mlmodern}

\begin{document}

\title[Sasakian manifolds and spin-c Killing spinors]{Sasakian manifolds and spin-c Killing spinors}

\author[Alejandro Gil-García]{Alejandro Gil-García \orcidlink{0000-0002-9370-241X}}
\address{Scuola Internazionale Superiore di Studi Avanzati, Trieste, Italy}
\email{agilgarc@sissa.it}

\author[C. S. Shahbazi]{C. S. Shahbazi \orcidlink{0000-0003-1185-9569}}
\address{Departamento de Matem\'aticas, Universidad UNED - Madrid, Reino de Espa\~na}
\email{cshahbazi@mat.uned.es}

\begin{abstract}
Using the theory of complex spinorial forms, we prove that an odd-dimensional Riemannian manifold admits a pure spin-c Killing spinor with a real Killing constant $\alpha \in \mathbb{R}^{\ast}$ if and only if it is $\alpha$-Sasakian, thereby obtaining an extension of a well-known result by A. Moroianu that, under the purity assumption, does not require simple connectivity or completeness.\bigskip

\medskip


\end{abstract}
 
\maketitle

%\setcounter{tocdepth}{1} %doesn't display subsections in TOC 

%\tableofcontents

% % % % % % % % % % % % % % % % % % % % % % % % % % % % % % % % % % % % % %
% % % % % % % % % % % % % % % % % % % % % % % % % % % % % % % % % % % % % %

\section{Introduction}

% % % % % % % % % % % % % % % % % % % % % % % % % % % % % % % % % % % % % %
% % % % % % % % % % % % % % % % % % % % % % % % % % % % % % % % % % % % % %

This note is devoted to proving the following result.

\begin{thm}
\label{thm:maintheorem}
An odd-dimensional spin-c Riemannian manifold admits a pure spin-c Killing spinor with real Killing constant $\alpha \in \mathbb{R}^{\ast}$ if and only if it is $\alpha$-Sasakian.
\end{thm}

This theorem extends, under the purity assumption, the odd-dimensional case of Moroianu's classification \cite[Thm.\ 4.1]{Moroianu1997} of spin-c Killing spinors, as well as the foundational results of C.\ B\"ar \cite{Bar93}, both of which require the underlying Riemannian manifold $(M,g)$ to be complete and simply connected. The proofs in these classical works rely on a Riemannian cone construction. Completeness is required to ensure via Gallot's theorem \cite{Gallot1979} that the cone is irreducible, while simply connectivity is assumed to guarantee that the holonomy group reduces to the restricted holonomy group and the local parallel spinor on the cone extends globally. Irreducibility of the cone implies, in turn, that the corresponding Killing spinor is pure, consistently connecting with the hypothesis in our theorem.\medskip

Our methods bypass the cone construction entirely and do not require either of these global topological assumptions. Instead, we use the differential theory of complex spinorial forms, developed in \cite{Algebraic_Complex_Square_2025,Complex_Differential_Spinors_2026}, to equivalently rephrase the spinorial equation in terms of an intrinsic differential system on the truncated K\"ahler-Atiyah bundle. By \cite[Thm.\ 5.10]{Complex_Differential_Spinors_2026}, solving this intrinsic differential equation on a Riemannian manifold admitting spin-c structures automatically guarantees the global existence of an irreducible complex Killing spinor with real Killing constant. The assumption that the spin-c Killing spinor is pure is \emph{sharp}: as pointed out to us by A.\ Moroianu in a private communication, a reducible cone can admit parallel spinors constructed as tensor products of spinors from its constituent factors. These product spinors may descend to the base manifold as strictly non-pure spin-c Killing spinors. Algebraically, the complex-bilinear square of such a non-pure spinor expands into an inhomogeneous exterior form, hence obstructing the construction of the contact one-form and the fundamental two-form required to obtain an almost contact metric structure.\medskip

The geometric significance of real Killing spinors was first established by T.\ Friedrich \cite{Friedrich1980}, who demonstrated that they arise as the eigenspinors minimizing the spectrum of the Dirac operator on compact spin manifolds. The subsequent classification of odd-dimensional manifolds admitting such spinors originated with the low-dimensional Sasaki-Einstein analyses of T.\ Friedrich and I.\ Kath \cite{FriedrichKath1989, FriedrichKath1990}, culminating in the aforementioned general classification by C.\ B\"ar and A.\ Moroianu \cite{Bar93,Moroianu1997}.\medskip

We expect that the methods used in this note can be applied to study the global geometry and topology of other related parallelicity conditions for irreducible complex spinors associated to both spin and spin-c structures.\medskip

Section \ref{sec:purespinors} is devoted to adapting the general formalism of \cite{Algebraic_Complex_Square_2025,Complex_Differential_Spinors_2026} to the case of irreducible pure complex spinors and the Killing equation in Riemannian signature. Theorem \ref{thm:maintheorem} is proven in Sections \ref{sec:onlyif} and \ref{sec:if}.

\begin{ack}
We thank D.\ Conti and A.\ Moroianu for illuminating discussions and comments. The work of AGG is supported by the Scuola Internazionale Superiore di Studi Avanzati (SISSA). The work of CSS was partially supported by the research grant PID2023-152822NB-I00 of the Ministry of Science of the government of Spain. 
\end{ack}

% % % % % % % % % % % % % % % % % % % % % % % % % % % % % % % % % % % % % %
% % % % % % % % % % % % % % % % % % % % % % % % % % % % % % % % % % % % % %

\section{Pure spinors in odd dimensions}
\label{sec:purespinors}

% % % % % % % % % % % % % % % % % % % % % % % % % % % % % % % % % % % % % %
% % % % % % % % % % % % % % % % % % % % % % % % % % % % % % % % % % % % % %

Let $(M,g)$ be an oriented Riemannian manifold of arbitrary odd dimension $d = 2n + 1$, and let $S$ be a bundle of irreducible complex spinors over the bundle of Clifford algebras of $(M,g)$. The existence of such a bundle of irreducible complex spinors is equivalent to the existence of a $\mathrm{Spin}^c(2n+1)$-structure $Q$ on $(M,g)$, to which $S$ is associated via the standard associated vector bundle construction \cite{LS18}:
\begin{equation*}
S = Q\times_{\mathfrak{r}_{\ell}} \mathbb{C}^{2^n},
\end{equation*}

\noindent
where $\mathfrak{r}_{\ell}\colon \Spin^c(2n+1) \to \mathrm{GL}(2^n,\C)$ denotes the restriction to the spin-c group $\Spin^c(2n+1) \subset \mCl(2n+1)$ of one of the two inequivalent irreducible representations $\gamma_{\ell}\colon\C\mathrm{l}(2n+1)\to\mathrm{End}(\smash{\mathbb{C}^{2^n}})$ of $\C\mathrm{l}(2n+1)$. These representations are distinguished by the sign $\ell\in\{\pm1\}$ that $\gamma_{\ell}$ takes on the complex volume form.  We refer the reader to \cite[Sec.~2.4]{Friedrich2000} for a detailed exposition of spin-c structures on Riemannian manifolds. In particular, $S$ has complex rank $2^n$ and admits a non-degenerate admissible complex-bilinear pairing $\scB$ \cite{AC97,ACDVP05}. This can be understood as a morphism of complex vector bundles:
\begin{equation*}
\scB \colon S\otimes S \to \mathcal{L},
\end{equation*}

\noindent
where $\mathcal{L}$ is the characteristic Hermitian complex line bundle on $M$ naturally associated with the given $\mathrm{Spin}^c(2n+1)$-structure $Q$. We will denote the corresponding principal $\U(1)$-bundle by $P\to M$. Since $S$ is associated to a spin-c structure $Q$, the combination of the Levi-Civita connection $\nabla^g$ of $(M,g)$ together with a choice of a connection $A\in\Omega^1(P,i\R)$ on the characteristic $\U(1)$-bundle $P$ defines a covariant derivative $\nabla^{g,A}$ on $S$. For ease of notation, in the following we will denote Clifford multiplication simply by a \emph{dot}.

\begin{definition}[{\cite[Chap.~IV, Sec.~9]{Spin89}}]
A nowhere vanishing section $\eta\in\Gamma(S)$ of $S$ is \emph{pure} if $\eta$ is pure at every point of $M$, that is, if at every point $p\in M$ the following subspace:
\begin{equation*}
\mathrm{Ann}(\eta_p):=\{X\in T_pM\otimes\C\mid X\cdot\eta_p=0\},
\end{equation*} 
which is automatically isotropic with respect to the complex-bilinear extension of $g$, is of maximal dimension, that is,  $\dim_\C\mathrm{Ann}(\eta_p)=n$ for every $p\in M$.
\end{definition}

\begin{definition} 
A \emph{pure spin-c Killing spinor} on $(M,g)$ with \emph{real Killing constant} $\alpha \in \mathbb{R}^{\ast}$ is a pure section $\eta\in \Gamma(S)$ of a bundle of irreducible complex spinors $S = Q  \times_{\mathfrak{r}_{\ell}} \C^{2^n}$ on $(M,g)$ satisfying the following equation:
\begin{equation}
\label{eq:KillingSpinor}
\nabla^{g,A}_X \eta = i\frac{\alpha}{2} X\cdot \eta\, , \qquad X \in \Gamma(TM)
\end{equation}

\noindent
for a connection $A\in\Omega^1(P,i\R)$ on the characteristic $\U(1)$-bundle $P$ of $Q$.
\end{definition}

\begin{remark}
Note the factor $i$ in the definition of a Killing spinor. Given that most of the mathematical literature in spin geometry uses the \emph{minus convention} to define the Clifford algebra, the inclusion of the factor $i$ ensures that our notion of real spin-c Killing spinor matches that appearing in the literature.
\end{remark}

\begin{remark}
If the connection $A$ has trivial holonomy, the principal $\mathrm{U}(1)$-bundle $P$ admits a global parallel section. In this gauge, $A$ globally vanishes, the spin-c structure strictly reduces to a classical spin structure, and the covariant derivative $\nabla^{g,A}$ descends to the standard spin Levi-Civita connection $\nabla^g$. Consequently, the spin-c Killing equation \eqref{eq:KillingSpinor} globally reduces to the standard Killing spinor equation. Interestingly enough, if $A$ is flat but has non-trivial holonomy, which requires $M$ to be non simply connected, then Equation \eqref{eq:KillingSpinor} reduces locally to the standard Killing spinor equation, but not globally. In this flat but holonomy non-trivial case, a spin-c Killing spinor can be intuitively understood as a standard Killing spinor \emph{twisted} by $\mathrm{U}(1)$ monodromies.
\end{remark}

Associated with a given complex spinor $\eta\in \Gamma(S)$ on $(M,g)$, we can construct its complex-bilinear square $\rho$, obtained via the choice of the admissible complex-bilinear pairing $\scB$ on $S$. The general algebraic characterization of this complex-bilinear square was obtained in \cite[Thm.\ 4.12]{Algebraic_Complex_Square_2025} in terms of an explicit algebraic system of equations in the underlying truncated Kähler-Atiyah bundle. Focusing on the case of a pure irreducible complex spinor, its complex-bilinear square is algebraically well-known since Cartan \cite{Cartan1938}, see also \cite{Kopczynski1997}. Globalizing this algebraic characterization to a Riemannian manifold equipped with a spin-c structure, see \cite{Complex_Differential_Spinors_2026}, we obtain the following result.

\begin{lemma}
\label{lemma:squaresriemannian2n1}
Let $(M,g)$ be a $(2n+1)$-dimensional oriented Riemannian manifold, and let $S$ be a bundle of irreducible complex spinors associated to a given $\mathrm{Spin}^c(2n+1)$-structure $Q$ with characteristic line bundle $\mathcal{L}$. A $\mathcal{L}$-valued complex exterior form $\rho \in \Omega_{\mathbb{C}}(M,\mathcal{L})$ is the complex-bilinear square of a pure irreducible complex spinor $\eta\in\Gamma(S)$ if and only if it is a maximally decomposable $\mathcal{L}$-valued complex $n$-form whose components are isotropic and mutually orthogonal. That is, locally around every point in $M$, there exists a local section $\mathfrak{l}$ of $\mathcal{L}$ such that:
\begin{equation*}
\rho = (\theta^1\wedge \cdots \wedge \theta^n) \otimes \mathfrak{l}
\end{equation*}

\noindent
for a set of local isotropic and mutually orthogonal complex one-forms $\theta^1,\ldots,\theta^n \in \Omega^1_{\mathbb{C}}(M)$.
\end{lemma}
 
Based on the previous lemma, we introduce the following definition. 
 
\begin{definition}
Let $(M,g)$ be a Riemannian manifold of dimension $d = 2n+1$ equipped with a Hermitian complex line bundle $\cL$. A \emph{Cartan $n$-form} on $(M,g,\cL)$ is a maximally decomposable $\mathcal{L}$-valued complex $n$-form whose components are isotropic and mutually orthogonal, as explained in Lemma \ref{lemma:squaresriemannian2n1}.
\end{definition}

Note that the spinor $\eta\in\Gamma(S)$ vanishes at a point $p\in M$ if and only if its complex-bilinear square $\rho$ also vanishes at $p$. By mapping the Killing spinor equation \eqref{eq:KillingSpinor} through the dequantization isomorphism $(\Psi_\ell^<)^{-1} \colon \mathrm{End}(S) \to \wedge^< T^*_\mathbb{C} M$ onto the truncated Kähler-Atiyah algebra $(\wedge^< T^*_{\mathbb{C}}M, \vee)$ of $(M,g)$, see \cite{Algebraic_Complex_Square_2025,Complex_Differential_Spinors_2026}, we translate the spinorial framework into an equivalent exterior system for $\rho$.\medskip

For later convenience, we will fix from now on $\ell=(-1)^{\binom{n-1}{2}+1}\in\{\pm1\}$.

\begin{prop}
\label{prop:KillingSpinorBilinear}
A $(2n+1)$-dimensional oriented and spin-c Riemannian manifold $(M,g)$ admits a pure spin-c Killing spinor with real Killing constant $\alpha \in \mathbb{R}^{\ast}$ if and only if it admits a Hermitian complex line bundle $\cL$ and a Cartan $n$-form $\rho \in \Omega^n_{\mathbb{C}}(M, \mathcal{L})$ satisfying the differential equation:
\begin{equation*}
\nabla^{g,A}_X \rho = i^{n+1}\alpha \ast (X^\flat \wedge \rho)
\end{equation*}

\noindent
for every vector field $X \in \Gamma(TM)$.
\end{prop}

\begin{proof}
Let $(M,g)$ be a $(2n+1)$-dimensional oriented and Riemannian manifold, equipped with a bundle of irreducible complex spinors $(S,\gamma_\ell)$. By \cite{Algebraic_Complex_Square_2025}, we can equip $(S,\gamma_\ell)$ with an admissible complex-bilinear pairing $\scB$ of positive adjoint type if $n$ is even, and negative adjoint type if $n$ is odd. As explained in \cite{LBC13,LBC16} and further elaborated in \cite{Algebraic_Complex_Square_2025}, the truncated Kähler-Atiyah bundle is defined on:
\begin{equation*}
\wedge^<T^*_\C M=\bigoplus_{k=0}^n\wedge^kT^*_\C M
\end{equation*}

\noindent
endowed with the product:
\begin{equation*}
\rho_1\vee\rho_2=\mathcal{P}_<(\rho_1\diamond\rho_2+i^n\ell*\tau(\rho_1\diamond\rho_2))
\end{equation*}

\noindent
for all $\rho_1,\rho_2\in\wedge^<T^*_\C M$ and where we have fixed $\ell=(-1)^{\binom{n-1}{2}+1}\in\{\pm1\}$. Here $\mathcal{P}_<\colon\wedge T^*_\C M\to\wedge^<T^*_\C M$ is the natural projection, $\diamond$ denotes the geometric product and $\tau$ is the reversion anti-automorphism that acts by $\smash{\tau(\beta)=(-1)^{\binom{k}{2}}\beta}$ on $k$-forms. By \cite[Thm.\ 5.10]{Complex_Differential_Spinors_2026}, the equation $\nabla^{g,A}_X \eta = i\frac{\alpha}{2} X\cdot \eta$ is equivalent to:
\begin{equation*}
\nabla^{g,A}_X\rho=i\frac{\alpha}{2}(X^\flat\vee\rho+s\rho\vee X^\flat),
\end{equation*}

\noindent
where $s\in\{\pm1\}$ is the adjoint type of $\scB$. We have: \begin{itemize}
    \item If $n$ is even, then $s=1$ and we get $X^\flat\vee\rho+\rho\vee X^\flat=2\ell*(X^\flat\wedge\rho)$.
    \item If $n$ is odd, then $s=-1$ and we get $X^\flat\vee\rho-\rho\vee X^\flat=-2i\ell*(X^\flat\wedge\rho)$.
\end{itemize}

Therefore, we obtain the result from the statement.
\end{proof}

% % % % % % % % % % % % % % % % % % % % % % % % % % % % % % % % % % % % % %
% % % % % % % % % % % % % % % % % % % % % % % % % % % % % % % % % % % % % %

\section{The \emph{only if} direction}
\label{sec:onlyif}

% % % % % % % % % % % % % % % % % % % % % % % % % % % % % % % % % % % % % %
% % % % % % % % % % % % % % % % % % % % % % % % % % % % % % % % % % % % % %

Assume that the $(2n+1)$-dimensional oriented Riemannian manifold $(M,g)$ admits a pure spin-c Killing spinor $\eta \in \Gamma(S)$ with real Killing constant $\alpha\in\mathbb{R}^{\ast}$. Let $\rho \in \Omega^n_{\mathbb{C}}(M, \mathcal{L})$ be its corresponding complex-bilinear square. We will construct a canonical $\alpha$-Sasakian structure on $(M,g)$ directly from $\rho$ via certain natural contractions and exterior products.\medskip

First, we construct the contact one-form, whose dual we will identify with the Reeb vector field of the $\alpha$-Sasakian structure. Taking the exterior product of $\rho$ with its complex conjugate $\overline{\rho}$, the line bundle factors drop out under the canonical pairing $\mathcal{L} \otimes \overline{{\mathcal{L}}}\cong\mathcal{L}\otimes\mathcal{L}^\ast\cong\mathbb{C}$, yielding a globally well-defined real $2n$-form $i^n\rho \wedge \overline{\rho} \in \Omega^{2n}(M)$. Applying the Hodge star operator, we define the contact one-form $\zeta \in \Omega^1(M)$ by:
\begin{equation*}
\zeta := \frac{i^n*(\rho\wedge\overline{\rho})}{\escal{\rho,\overline{\rho}}},
\end{equation*}

\noindent
where $\langle \rho, \overline{\rho} \rangle \in C^{\infty}(M)$ is the inner product defined by $g$. The dual $\xi := \zeta^\sharp$ of $\zeta$ is a nowhere vanishing vector field on $M$ of unit length, since $\rho$ is nowhere vanishing and:
\begin{equation*}
\escal{*(\rho\wedge\overline{\rho}),*(\rho\wedge\overline{\rho})}=(-1)^n\escal{\rho,\overline{\rho}}^2.
\end{equation*}

\noindent
Next, we construct the fundamental two-form $\omega \in \Omega^2(M)$ out of $\rho$. Given $\rho\in \Omega^n_{\mathbb{C}}(M,\cL)$, define:
\begin{equation*}
\rho \triangle_{n-1} \overline{\rho} := \frac{1}{(n-1)!} \rho(e_{i_1},\hdots , e_{i_{n-1}}) \wedge \overline{\rho}(e_{i_1},\hdots , e_{i_{n-1}})\in \Omega^2_{\mathbb{C}}(M)
\end{equation*}

\noindent
in terms of any orthonormal frame $\{e_1,\hdots , e_{d}\}$ of $(M,g)$. We set:
\begin{equation*}
i\omega := \frac{\rho\triangle_{n-1}\overline{\rho}}{\escal{\rho,\overline{\rho}}}.
\end{equation*}
Hence, $\omega \in \Omega^2(M)$ is a real two-form on $M$.

\begin{lemma}\label{lemma:nabla_zeta}
The contact one-form $\zeta\in\Omega^1(M)$ satisfies $\nabla^g\zeta = -\alpha\omega$. In particular, the Reeb vector field $\xi$ is Killing.
\end{lemma}

\begin{proof}
Let $c_n:=i^{n+1}\alpha$. Then, by Proposition \ref{prop:KillingSpinorBilinear}, we have $\nabla^{g,A}_X\rho=c_n*(X^\flat\wedge\rho)$ for all $X\in\Gamma(TM)$, thus:
\begin{equation*}
X\escal{\rho,\overline{\rho}}=c_n\escal{*(X^\flat\wedge\rho),\overline{\rho}}+\overline{c}_n\escal{\rho,*(X^\flat\wedge\overline{\rho})}.
\end{equation*}

\noindent
Note that:
\begin{equation*}
\escal{*(X^\flat\wedge\rho),\overline{\rho}}=\escal{X^\flat\wedge\rho,*\overline{\rho}}=\escal{\rho,\iota_X(*\overline{\rho})}=\escal{\rho,*(\overline{\rho}\wedge X^\flat)}=(-1)^n\escal{\rho,*(X^\flat\wedge\overline{\rho})}.
\end{equation*}

\noindent
Hence $X\escal{\rho,\overline{\rho}}=((-1)^nc_n+\overline{c}_n)\escal{\rho,*(X^\flat\wedge\overline{\rho})}=0$ for all $n\in\N$ since $\overline{c}_n=(-1)^{n+1}c_n$. Using this we obtain:
\begin{equation*}
\nabla^{g}_X\zeta=c_n\frac{i^n}{\escal{\rho,\overline{\rho}}}*\big(*(X^\flat\wedge\rho)\wedge\overline{\rho} - *(X^\flat\wedge\overline{\rho})\wedge\rho\big).
\end{equation*}

\noindent
Now we compute: \begin{align*}
    T(X,Y)&:=\iota_Y(*(*(X^\flat\wedge\rho)\wedge\overline{\rho}))=*(*(X^\flat\wedge\rho)\wedge\overline{\rho}\wedge Y^\flat)\\
    &{\phantom{:}}=(-1)^n*(Y^\flat\wedge\overline{\rho}\wedge*(X^\flat\wedge\rho))=(-1)^n*\escal{Y^\flat\wedge\overline{\rho},X^\flat\wedge\rho}\nu\\
    &{\phantom{:}}=(-1)^n\escal{X^\flat\wedge\rho,Y^\flat\wedge\overline{\rho}}=(-1)^n\big(g(X,Y)\escal{\rho,\overline{\rho}}-\escal{\iota_X\overline{\rho},\iota_Y\rho}\big),
\end{align*}

\noindent
thus $T(X,Y)-\overline{T(X,Y)}=(-1)^n(\escal{\iota_X\rho,\iota_Y\overline{\rho}}-\escal{\iota_X\overline{\rho},\iota_Y\rho})$. On the other hand, it can be seen that: $$(\rho\triangle_{n-1}\overline{\rho})(X,Y)=\escal{\iota_X\rho,\iota_Y\overline{\rho}}-\escal{\iota_X\overline{\rho},\iota_Y\rho}.$$

\noindent
Therefore, we obtain the formula in the statement.\medskip

From the formula for $\nabla^g\zeta$ we deduce that $\mathscr{L}_{\xi}g = \mathrm{Sym}(\nabla^g\zeta)=0$, hence $\xi$ is a Killing vector field.
\end{proof}

We define the fundamental $(1,1)$-tensor field $\varphi\in\Gamma(\mathrm{End}(TM))$ by:
\begin{equation*}
\varphi(X) := -\frac{1}{\alpha}\nabla_X^g\xi,
\end{equation*}

\noindent
which satisfies $g(\varphi(X),Y)=\omega(X,Y)$ for all $X,Y\in\Gamma(TM)$ and then $g(\varphi(X),Y)=-g(X,\varphi(Y))$.

\begin{lemma}\label{lemma:varphi_square}
The fundamental endomorphism $\varphi\in\Gamma(\End(TM))$ satisfies $\varphi^2=-\Id+\zeta\otimes\xi$.
\end{lemma}

\begin{proof}
We first show that $\iota_\xi\rho=0$. We compute:
\begin{equation*}
\iota_\xi(\rho\wedge\overline{\rho})=\iota_\xi(*(*(\rho\wedge\overline{\rho})))=*(*(\rho\wedge\overline{\rho})\wedge\zeta)=i^{-n}\escal{\rho,\overline{\rho}}*(\zeta\wedge\zeta)=0.
\end{equation*}

\noindent
This implies that $\iota_\xi\rho=\iota_\xi\overline{\rho}=0$ since $\rho\wedge\overline{\rho}$ is maximally decomposable.\medskip

Now let $\{e_1,\ldots,e_{2n},\xi\}$ be an orthonormal frame of $(M,g)$ and denote by $\{e^1,\ldots,e^{2n},\zeta\}$ the dual coframe. Set $\theta^j:=e^{2j-1}+ie^{2j}$ for $j=1,\ldots,n$. Then $\rho=(\theta^1 \wedge \cdots \wedge \theta^n)\otimes\mathfrak{l}$ for a local unit section $\mathfrak{l}$ of $\mathcal{L}$. Since $\escal{\theta^j,\theta^k}=0$ and $\escal{\theta^j,\bar{\theta}^k}=2\delta_{jk}$ for all $j,k$, we have $\escal{\rho,\overline{\rho}}=2^n$. Using the definition of $\rho \triangle_{n-1}\overline{\rho}$ we obtain:
\begin{equation*}
\rho\triangle_{n-1}\overline{\rho}=2^{n-1}\sum_{j=1}^n\theta^j\wedge\bar{\theta}^j .
\end{equation*}

\noindent
On the other hand, applying the definition of $\omega$ yields:
\begin{equation*}
\omega=-i\frac{\rho\triangle_{n-1}\overline{\rho}}{\escal{\rho,\overline{\rho}}}=-\frac{i}{2}\sum_{j=1}^n\theta^j\wedge\bar{\theta}^j=-\sum_{j=1}^n e^{2j-1}\wedge e^{2j}.
\end{equation*}

\noindent
From $\varphi(X)=(\iota_X\omega)^\sharp$ we obtain $\varphi(e_{2j-1})=-e_{2j}$, $\varphi(e_{2j})=e_{2j-1}$, and $\varphi(\xi)=0$. Hence $\varphi^2(e_{2j-1})=-e_{2j-1}$, $\varphi^2(e_{2j})=-e_{2j}$, and $\varphi^2(\xi)=0$. Therefore, if $X=X_{\mathcal{H}}+\zeta(X)\xi$ with $X_{\mathcal{H}}\in\mathcal{H}:=\ker(\zeta)$, then we obtain:
\begin{equation*}
\varphi^2(X)=-X_{\mathcal{H}}=-X+\zeta(X)\xi,
\end{equation*}

\noindent
which proves the claim from the statement.
\end{proof}

\begin{lemma}\label{lemma:nabla_omega}
The fundamental two-form $\omega\in\Omega^2(M)$ satisfies $\nabla^g_X\omega=\alpha X^\flat\wedge\zeta$.
\end{lemma}

\begin{proof}
From Lemma \ref{lemma:nabla_zeta} we get $\nabla^g_X\zeta=-\alpha\iota_X\omega$, thus $\nabla^g_X\xi=(\nabla^g_X\zeta)^\sharp=-\alpha(\iota_X\omega)^\sharp=-\alpha\varphi(X)$. Using $\iota_{\xi}\omega=0$ we compute:
\begin{align*}
(\nabla^g_X\omega)(Y,\xi)&=-\omega(Y,\nabla^g_X\xi)=\alpha\omega(Y,\varphi(X))\\
&=\alpha g(\varphi(X),\varphi(Y))=\alpha\big(g(X,Y)-\zeta(X)\zeta(Y)\big),
\end{align*}

\noindent
where the last equality follows from $\varphi^2=-\Id+\zeta\otimes\xi$, see Lemma \ref{lemma:varphi_square}. On the other hand, it is clear that $(X^\flat\wedge\zeta)(Y,\xi)=g(X,Y)-\zeta(X)\zeta(Y)$. Hence, we obtain the expression stated.\medskip

It remains to show that $(\nabla^g_X\omega)(Y,Z)=0$ for all $Y,Z\in\mathcal{H}=\ker(\zeta)$. First, let us consider $X=\xi$. Since $i^n\rho\wedge\overline{\rho}\wedge\zeta=(-1)^{\binom{n}{2}}2^n\nu$ and $*\rho=(-1)^{\binom{n+1}{2}}i^n\rho\wedge\zeta$ we have $*(\zeta\wedge\rho)=(-1)^{\binom{n+1}{2}}i^n\rho$ and then:
\begin{equation*}
\nabla^{g,A}_\xi\rho=i^{n+1}\alpha*(\zeta\wedge\rho)=(-1)^{\binom{n}{2}}i\alpha\rho
\end{equation*}

\noindent
by Proposition \ref{prop:KillingSpinorBilinear}. Using $\nabla^{g,A}_\xi(\rho\triangle_{n-1}\overline{\rho})=(\nabla^{g,A}_\xi\rho)\triangle_{n-1}\overline{\rho}+\rho\triangle_{n-1}(\nabla^{g,A}_\xi\overline{\rho})$ we conclude that $\nabla^g_\xi\omega=0$. Now assume that $X,Y,Z\in\mathcal{H}$. Then:
\begin{equation*}
*(X^\flat\wedge\rho)=(-1)^n\iota_X(*\rho)=(-1)^{\binom{n}{2}}i^n\iota_X(\rho\wedge\zeta)=(-1)^{\binom{n}{2}}i^n\iota_X\rho\wedge\zeta
\end{equation*}

\noindent
and $\nabla^{g,A}_X\rho=(-1)^{\binom{n+1}{2}}i\alpha\iota_X\rho\wedge\zeta$. Since $\iota_{\xi}\rho=\iota_{\xi}\overline{\rho}=0$ we have:
\begin{align*}
\escal{\iota_Y(\iota_X\rho\wedge\zeta),\iota_Z\overline{\rho}}&=\escal{\iota_Y\iota_X\rho\wedge\zeta,\iota_Z\overline{\rho}}=(-1)^{n-2}\escal{\zeta\wedge\iota_Y\iota_X\rho,\iota_Z\overline{\rho}}\\
&=(-1)^{n-2}\escal{\iota_Y\iota_X\rho,\iota_\xi\iota_Z\overline{\rho}}=(-1)^{n-1} \escal{\iota_Y\iota_X\rho ,\iota_Z\iota_\xi \overline{\rho}}=0.
\end{align*}

\noindent
A straightforward computation using $\omega=(i\escal{\rho,\overline{\rho}})^{-1}\rho\triangle_{n-1}\overline{\rho}$, $(\rho\triangle_{n-1}\overline{\rho})(Y,Z)=\escal{\iota_Y\rho,\iota_Z\overline{\rho}}-\escal{\iota_Y\overline{\rho},\iota_Z\rho}$, and the above result we conclude that $(\nabla^g_X\omega)(Y,Z)=0$ for all $X,Y,Z\in\mathcal{H}$.
\end{proof}

The formula in Lemma \ref{lemma:nabla_omega} is equivalent to $(\nabla^g_X\varphi)Y=\alpha(g(X,Y)\xi-\zeta(Y)X)$, which is the integrability condition for the tuple $(\xi, \zeta, \varphi, g)$ to constitute a well-defined $\alpha$-Sasakian structure on $M$. Therefore, we have proved the following result.

\begin{prop}
Let $(M,g)$ be a $(2n+1)$-dimensional oriented and spin-c Riemannian manifold admitting a pure spin-c Killing spinor with real Killing constant $\alpha\in\mathbb{R}^{\ast}$. Then $(M,g)$ is $\alpha$-Sasakian.
\end{prop}

% % % % % % % % % % % % % % % % % % % % % % % % % % % % % % % % % % % % % %
% % % % % % % % % % % % % % % % % % % % % % % % % % % % % % % % % % % % % %

\section{The \emph{if} direction}
\label{sec:if}

% % % % % % % % % % % % % % % % % % % % % % % % % % % % % % % % % % % % % %
% % % % % % % % % % % % % % % % % % % % % % % % % % % % % % % % % % % % % %

Assume that $(M,g)$ is a $(2n+1)$-dimensional $\alpha$-Sasakian manifold. By definition, $(M,g)$ admits an almost contact metric structure $(\xi, \zeta, \varphi, g)$ where $\xi$ is the Reeb vector field, $\zeta = \xi^\flat$ is the contact one-form, and $\varphi$ is the fundamental $(1,1)$-tensor field acting as an isometry on the contact distribution $\mathcal{H}  = \ker(\zeta)$ and satisfying $\varphi^2 = -\mathrm{Id} + \zeta \otimes \xi$. The $\alpha$-Sasakian condition amounts to the following equations:
\begin{equation*}
\nabla^g_X \zeta = -\alpha\, \iota_X \omega,\qquad \nabla^g_X \omega = \alpha\, X^\flat \wedge \zeta,
\end{equation*}

\noindent
where $\alpha\in\mathbb{R}^{\ast}$, $\nabla^g$ is the Levi-Civita connection, and $\omega(X,Y) = g(\varphi(X),Y)$ for $X,Y\in\Gamma(TM)$ is the fundamental two-form of the almost contact metric structure. The restriction:
\begin{equation*}
J = \varepsilon_n\varphi\vert_{\mathcal{H}},\qquad \varepsilon_n:=(-1)^{\binom{n}{2}+1}\in\{\pm1\}
\end{equation*}

\noindent
defines an almost complex structure on the real rank-$2n$ distribution $\mathcal{H} \subset TM$, which induces a splitting of its complexification into holomorphic and anti-holomorphic subbundles:
\begin{equation*}
\mathcal{H} _{\mathbb{C}} := \mathcal{H}\otimes \mathbb{C} = \mathcal{H} ^{(1,0)} \oplus \mathcal{H} ^{(0,1)}.
\end{equation*}

\noindent
Since $\mathcal{H} ^{(1,0)}$ is a complex vector bundle of rank $n$, its top exterior power defines a complex line bundle over $M$. We identify the characteristic line bundle $\mathcal{L}$ of the underlying $\mathrm{Spin}^c(2n+1)$-structure with this determinant bundle, setting $\mathcal{L} \cong \wedge^n_{\mathbb{C}} \mathcal{H} ^{(1,0)}$.

\begin{remark}
Every $\alpha$-Sasakian manifold $M^{2n+1}$ is, in particular, an almost contact metric manifold. Hence $TM$ admits a reduction of the structure group from $\mathrm{SO}(2n+1)$ to $\mathrm{U}(n)$. Since the standard inclusion $\mathrm{U}(n)\subset\mathrm{SO}(2n)\subset \mathrm{SO}(2n+1)$ admits a canonical lift $\mathrm{U}(n)\to\mathrm{Spin}^{c}(2n+1)$, see e.g.\ \cite{Spin89}, this reduction induces a canonical spin-c structure on $M$.
\end{remark}

Let $\{e_1,\ldots,e_{2n},e_{2n+1}=\xi\}$ be a local adapted orthonormal frame on $(M,g)$ such that $J(e_{2j-1}) = e_{2j}$ and $J(e_{2j}) = -e_{2j-1}$ for $j=1, \ldots, n$, with corresponding dual coframe $\{e^1,\ldots,e^{2n},e^{2n+1}=\zeta\}$. Under this choice of frame, a local frame for $\mathcal{H}^{(1,0)}$ is given by the complex vector fields:
\begin{equation*}
Z_j := \frac{1}{2}(e_{2j-1} - ie_{2j}), \qquad j=1, \ldots, n
\end{equation*}

\noindent
and a local frame for the dual bundle $\mathcal{H}^{(1,0)\ast}$ is given by the complex one-forms:
\begin{equation*}
\theta^j := e^{2j-1} + ie^{2j}, \qquad j=1, \ldots, n.
\end{equation*}

\noindent
We define the $\mathcal{L}$-valued complex exterior $n$-form $\rho \in \Omega^n_{\mathbb{C}}(M, \mathcal{L})$ as follows:
\begin{equation*}
\rho := \theta \otimes \mathfrak{l} := (e^1 + i e^2) \wedge \cdots \wedge (e^{2n-1} + i e^{2n}) \otimes \mathfrak{l},
\end{equation*}

\noindent
where, using the bundle isomorphism $\mathcal{L} \cong \wedge^n_{\mathbb{C}} \mathcal{H} ^{(1,0)}$, the local non-vanishing section $\mathfrak{l} \in \Gamma(\mathcal{L})$ is explicitly identified in terms of the underlying real local orthonormal frame of $(M,g)$ via the complexified exterior product as:
\begin{equation*}
\mathfrak{l} = (e_1 - i e_2) \wedge \cdots \wedge (e_{2n-1} - i e_{2n}) = 2^n Z_1 \wedge \cdots \wedge Z_n \in \Gamma(\wedge^n_{\mathbb{C}} \mathcal{H} ^{(1,0)}).
\end{equation*}

\begin{lemma}
The $\mathcal{L}$-valued complex exterior form $\rho\in\Omega^n_\C(M,\mathcal{L})$ is globally well-defined.
\end{lemma}

\begin{proof}
Under a local change of the horizontal frame by a transition matrix $U \in \mathrm{U}(n)$ such that:
\begin{equation*}
\tilde{Z}_j = \sum_k U_{jk} Z_k    
\end{equation*}

\noindent
the top-degree vector field section transforms as $\tilde{\mathfrak{l}} = \det(U) \mathfrak{l}$. By duality and unitarity, the corresponding coframe forms transform under the complex-conjugate matrix, yielding $\tilde{\theta} = \det(\overline{U})\theta = \overline{\det(U)}\theta$. Their variations cancel out identically under the tensor product since $\overline{\det(U)}\det(U) = 1$, whence $\rho$ becomes a globally well-defined, locally decomposable nowhere vanishing $n$-form on $M$ taking values in $\mathcal{L}$.
\end{proof}

Let $\mathcal{P}_{\mathcal{H}} \colon TM \to \mathcal{H}$ be the orthogonal projection onto the contact distribution $\mathcal{H} = \ker(\zeta)$, defined explicitly by $\mathcal{P}_{\mathcal{H}}(X) = X - \zeta(X)\xi$ for every $X\in TM$. We define a connection $\nabla^{\mathcal{H}}$ on the real vector bundle $\mathcal{H} \to M$ by projecting the Levi-Civita connection $\nabla^g$:
\begin{equation*}
\nabla^{\mathcal{H}}_X Y := \mathcal{P}_{\mathcal{H}}(\nabla^g_X Y)
\end{equation*}

\noindent
for every $Y \in \Gamma(\mathcal{H})$ and $X\in \Gamma(TM)$. Since $(M,g)$ is $\alpha$-Sasakian, the restricted endomorphism $J = \varepsilon_n\varphi\vert_{\mathcal{H}}$ acts as a parallel complex structure with respect to $\nabla^{\mathcal{H}}$, meaning $\nabla^{\mathcal{H}}_X (JY) = J \nabla^{\mathcal{H}}_X Y$. Extending $\nabla^{\mathcal{H}}$ complex-linearly to the complexified bundle $\mathcal{H}_{\mathbb{C}}$, it preserves the splitting: 
\begin{equation*}
\mathcal{H}_{\mathbb{C}} = \mathcal{H}^{(1,0)} \oplus \mathcal{H}^{(0,1)}
\end{equation*}

\noindent
and restricts to a canonical connection on the subbundle $\mathcal{H}^{(1,0)}$. This connection naturally induces a connection $\nabla^{\mathrm{Det}}$ on the complex determinant line bundle $\mathcal{L} = \wedge^n_{\mathbb{C}} \mathcal{H}^{(1,0)}$, which acts on the local section $\mathfrak{l}$ via the trace of the horizontal connection forms:
\begin{equation*}
\nabla^{\mathrm{Det}}_X \mathfrak{l} = i \sum_{j=1}^n \varpi_{2j-1,2j}(X) \mathfrak{l},
\end{equation*}

\noindent
where $\varpi_{ab}(X):=g(\nabla^g_Xe_a,e_b)$. The Hermitian connection $\nabla^A$ on $\mathcal{L}$ that we shall use to define the connection $\nabla^{g,A}$ occurring in the Killing spinor equation is defined as the vertical deformation of this induced connection:
\begin{equation*}
\nabla^A_X := \nabla^{\mathrm{Det}}_X + (-1)^{\binom{n}{2}}i\alpha \zeta(X) \mathrm{Id}_{\mathcal{L}}.
\end{equation*}

We now compute the total covariant derivative of $\rho \in\Omega^n_\C(M,\mathcal{L})$ with respect to the twisted connection $\nabla^{g,A} = \nabla^g \otimes \nabla^A$. Locally we have $\rho = \theta\otimes\mathfrak{l}$ with $\theta\in \Omega^n_{\C}(M)$, and by the Leibniz rule, we obtain:
\begin{equation*}
\nabla^{g,A}_X \rho = (\nabla^g_X \theta) \otimes \mathfrak{l} + \theta \otimes \nabla^A_X \mathfrak{l}.
\end{equation*}

The $\alpha$-Sasakian structure equations imply that the Levi-Civita connection mixed components satisfy:
\begin{equation*}
\varpi_{a, 2n+1}(X) = g(\nabla^g_X e_a, \xi) = \alpha (\iota_X \omega)(e_a)=\alpha g(\varphi(X),e_a)=\varepsilon_n\alpha g(J(X),e_a).
\end{equation*}

\noindent
For the complex one-forms $\theta^j$, evaluating the vertical projection yields the explicit relation:
\begin{equation*}
\nabla^g_X \theta^j = \mathcal{P}_{\mathcal{H}}(\nabla^g_X \theta^j)  + i \varepsilon_n\alpha \theta^j(X) \zeta.
\end{equation*}

\noindent
Extending this action to the full wedge product $\theta = \theta^1 \wedge \cdots \wedge \theta^n$, the vertical components add to yield:
\begin{equation*}
(\nabla^g_X \theta)_{\text{vert}} = i \varepsilon_n\alpha \sum_{j=1}^n \theta^1 \wedge \cdots \wedge (\theta^j(X)\zeta) \wedge \cdots \wedge \theta^n = i \varepsilon_n\alpha \zeta \wedge \iota_X \theta.
\end{equation*}

\noindent
Since $\mathfrak{l}$ is the top exterior product of the $(1,0)$-vectors dual, up to the factor $2^n$, to the local coframe $\{\theta^j\}_{j=1}^n$, its covariant derivative produces the trace contribution opposite to that of $\theta=\theta^1\wedge\cdots\wedge\theta^n$. Evaluating $\nabla^A$ on the local section $\mathfrak{l}$ yields:
\begin{equation*}
\nabla^A_X \mathfrak{l} = \Big(i\sum_{j=1}^n \varpi_{2j-1, 2j}(X) + (-1)^{\binom{n}{2}}i\alpha\zeta(X)\Big)\mathfrak{l}.
\end{equation*}

\noindent
When evaluating the total covariant derivative $\nabla^{g,A}_X \rho$, the local horizontal trace connection forms $i\sum_{j=1}^n \varpi_{2j-1, 2j}(X)$ from $\nabla^A_X \mathfrak{l}$ and $-i\sum_{j=1}^n \varpi_{2j-1, 2j}(X)$ from $\nabla^g_X \theta$ cancel out identically. Then we obtain:
\begin{align*}
\nabla^{g,A}_X\rho&=(\nabla^g_X\theta)\otimes\mathfrak{l}+\theta\otimes\nabla^A_X\mathfrak{l}\\
&=\Big(-i\sum_{j=1}^n\varpi_{2j-1,2j}(X)\theta+i\varepsilon_n\alpha\zeta\wedge\iota_X\theta\Big)\otimes\mathfrak{l}+\theta\otimes\Big(i\sum_{j=1}^n\varpi_{2j-1,2j}(X)+(-1)^{\binom{n}{2}}i\alpha\zeta(X)\Big)\mathfrak{l}\\
&=i\varepsilon_n\alpha(\zeta\wedge\iota_X\theta)\otimes\mathfrak{l}+(-1)^{\binom{n}{2}}i\alpha\zeta(X)\rho.
\end{align*}

\begin{lemma}
The $\mathcal{L}$-valued complex exterior form $\rho\in\Omega^n_\C(M,\mathcal{L})$ satisfies $\nabla^{g,A}_X \rho = i^{n+1}\alpha \ast (X^\flat \wedge \rho)$.
\end{lemma}

\begin{proof}
We have seen that:
\begin{equation}
\label{eq:cov_deriv}
\nabla_{X}^{g,A}\rho = i\varepsilon_n\alpha(\zeta \wedge \iota_{X}\theta) \otimes \mathfrak{l} + (-1)^{\binom{n}{2}}i\alpha\zeta(X)\rho.
\end{equation}

\noindent
Decomposing the vector field $X$ into its horizontal and Reeb components, $X = X_{\mathcal{H}} + \zeta(X)\xi$, we have the dual decomposition $X^{\flat} = X_{\mathcal{H}}^{\flat} + \zeta(X)\zeta$. Applying the Hodge star operator linearly, we get:
\begin{equation*}
*(X^{\flat} \wedge \rho) = *(X_{\mathcal{H}}^{\flat} \wedge \theta)\otimes \mathfrak{l} + \zeta(X)*(\zeta \wedge \theta)\otimes \mathfrak{l}.
\end{equation*}

\noindent
For the horizontal component, on a $(2n+1)$-dimensional contact metric manifold with volume form $\nu=\nu_{\mathcal{H}}\wedge\zeta$, the Hodge star operator reduces to the $2n$-dimensional Hodge star operator on the contact distribution $\mathcal{H}=\ker(\zeta)$. Using that:
\begin{equation*}
*_{\mathcal{H}}\theta=(-1)^{\binom{n+1}{2}}i^n\theta
\end{equation*}

\noindent
for the $(n,0)$-form $\theta$ we get:
\begin{equation*}
*(X_{\mathcal{H}}^{\flat} \wedge \theta) = *_{\mathcal{H}}(X_{\mathcal{H}}^{\flat} \wedge \theta) \wedge \zeta = (-1)^n\iota_{X_{\mathcal{H}}}(*_{\mathcal{H}}\theta)\wedge\zeta = (-1)^{\binom{n}{2}}i^n \iota_{X_{\mathcal{H}}} \theta \wedge \zeta.
\end{equation*}

\noindent
Multiplying by the overall constant $i^{n+1}\alpha$, the horizontal term becomes:
\begin{equation*}
i^{n+1}\alpha *(X_{\mathcal{H}}^{\flat} \wedge \theta) = i^{n+1}\alpha (-1)^{\binom{n}{2}}i^n \iota_{X_{\mathcal{H}}} \theta \wedge \zeta = i\varepsilon_n\alpha(\zeta\wedge\iota_X\theta),
\end{equation*}

\noindent
where we have used the fact that $\iota_{X_{\mathcal{H}}}\theta = \iota_{X}\theta$ since $\theta$ is purely horizontal.\medskip

For the vertical component, we have:
\begin{equation*}
*(\zeta\wedge\theta)=*_{\mathcal{H}}\theta=(-1)^{\binom{n+1}{2}}i^n\theta.
\end{equation*}

\noindent
Multiplying it by the constant $i^{n+1}\alpha$ we obtain:
\begin{equation*}
i^{n+1}\alpha\zeta(X)*(\zeta\wedge\theta)=(-1)^{\binom{n+1}{2}}i^ni^{n+1}\alpha\zeta(X)\theta=(-1)^{\binom{n}{2}}i\alpha\zeta(X)\theta.
\end{equation*}

\noindent
Combining the horizontal and vertical terms we get:
\begin{equation*}
i^{n+1}\alpha *(X^{\flat} \wedge \rho) = i\varepsilon_n\alpha(\zeta \wedge \iota_{X}\theta) \otimes \mathfrak{l} + (-1)^{\binom{n}{2}}i\alpha\zeta(X)\rho,
\end{equation*}

\noindent
which matches the expression for $\nabla_{X}^{g,A}\rho$ in Equation \eqref{eq:cov_deriv}, thereby concluding the proof.
\end{proof}

Therefore, by Proposition \ref{prop:KillingSpinorBilinear}, we have the following result.

\begin{prop}
Let $(M,g)$ be a $(2n+1)$-dimensional $\alpha$-Sasakian manifold. Then $(M,g)$ admits a pure spin-c Killing spinor with real Killing constant $\alpha\in\mathbb{R}^{\ast}$.
\end{prop}

% % % % % % % % % % % % % % % % % % % % % % % % % % % % % % % % % % % % % %
% % % % % % % % % % % % % % % % % % % % % % % % % % % % % % % % % % % % % %

\bibliographystyle{myamsplain}
\bibliography{biblio}

\end{document}